\documentclass[11pt]{article}
\usepackage{amsfonts,float}
\usepackage{mathrsfs}
\usepackage{amsmath,amssymb}
\usepackage{mathtools}

\usepackage[latin1]{inputenc}
\usepackage{amsmath,amssymb}
\usepackage{latexsym}
\usepackage{epstopdf}
\usepackage{geometry}   
\usepackage[active]{srcltx}
\usepackage{graphicx}
\usepackage{epstopdf}
\usepackage[
bookmarks=true,         
bookmarksnumbered=true, 
colorlinks=true, pdfstartview=FitV, linkcolor=blue, citecolor=blue,
urlcolor=blue]{hyperref}

\newtheorem {theorem}{Theorem}[section]

\newtheorem{definition}{Definition}[section]

\newenvironment{proof}[1][Proof]{\textbf{#1.} }{\
\rule{0.5em}{0.5em}}

\begin{document}

\title{On mutual information estimation for mixed-pair random variables}

\maketitle
\begin{center}
\bigskip
Aleksandr Beknazaryan,
Xin Dang
and Hailin Sang
$^{1}$\footnotetext[1]
{Corresponding author.}

\bigskip Department of Mathematics, The University of Mississippi,
University, MS 38677, USA. E-mail: abeknaza@olemiss.edu, xdang@olemiss.edu, sang@olemiss.edu

\end{center}
\begin{center}
\bigskip
\textbf{Abstract}
\end{center}
We study the mutual information estimation for mixed-pair random variables. One random variable is discrete and the other one is continuous. We develop a kernel method to estimate the mutual information between the two random variables. The estimates enjoy a central limit theorem under some regular conditions on the distributions. The theoretical results are demonstrated by simulation study. 



\section{Introduction}

The entropy of a discrete random variable $X\in \mathbb{R}^d$ with countable support $\{x_1,x_2,...\}$ and $p_i=\mathbb{P}(X=x_i)$ is defined to be
$$H(X)=-\sum\limits_{i}p_i\log p_i,$$
and the (differential) entropy of a continuous random variable $Y\in \mathbb{R}^d$ with probability density function $f(y)$ is defined as 
$$H(Y)=-\int_{\mathbb{R}^d}f(y)\log f(y)dy.$$

If $d\ge 2$, $H(X)$ or $H(Y)$ is also called the joint entropy of the components in $X$ or $Y$. Entropy is a measure of distribution uncertainty and naturally it has application in the fields of information theory, statistical classification, pattern recognition and so on. 

 Let $P_{X}$, $P_{Y}$ be probability measures on some arbitrary measure spaces $\mathcal{X}$  and  $\mathcal{Y}$ respectively. Let $P_{XY}$ be the joint probability measure on the space $\mathcal{X}\times \mathcal{Y}$.  If  $P_{XY}$ is absolutely continuous with respect to the product measure $P_X\times P_Y$, let $\frac{dP_{XY}}{d(P_X\times P_Y)}$ be the Radon-Nikodym derivative. Then the general definition of  the mutual information (e.g.,  \cite{Gao}) is given by
 \begin{equation}\label{MIdef}
 I(X,Y)=\int _{\mathcal{X}\times \mathcal{Y}}dP_{XY}\log \frac{dP_{XY}}{d(P_X\times P_Y)}. 
 \end{equation}
If two random variables $X$ and $Y$ are either both discrete or both continuous then the mutual information of $X$ and $Y$ can be  expressed in terms of entropies as 
\begin{equation}\label{3H}
I(X,Y)=H(X)+H(Y)-H(X,Y).
\end{equation}

However, in practice and application, we often need to work on a mixture of continuous and discrete random variables. There are several ways for the mixture. 1). One random variable $X$ is discrete and the other random variable $Y$ is continuous; 2). A random variable $Z$ has both discrete and continuous components, i.e., $Z=X$ with probability $p$ and $Z=Y$ with probability $1-p,$ where $0<p<1$, $X$ is a  discrete random variable and $Y$ is a continuous random variable; 3). a random vector with each dimension component being discrete, continuous or mixture as in 2).  

In \cite{Nair}, the authors extend the definition of the joint entropy for the first case mixture, i.e., for the pair of random variables, where the first random variable is discrete and the second one is continuous. 
%
Our goal is to study the mutual information for that case
and provide the estimation of the mutual information from a given i.i.d. sample $\{X_i,Y_i\}_{i=1}^N$. 

In \cite{Gao}, the authors applied the $k$-nearest neighbor method to estimate the Radon-Nikodym derivative and, therefore, to estimate the mutual information for all three mixed cases. 
In the literature,  if the random variables $X$ and $Y$ are either both discrete or both continuous, the estimation of mutual information is usually performed by the estimation of the three entropies in (\ref{3H}). The estimation of a differential entropy has been well studied. An incomplete list of the related research includes the nearest-neighbor estimator 
\cite{KL},
\cite{Tv},
\cite{LPS};   the kernel estimator 
\cite{AL},
\cite{Joe},
\cite{Hall1},
\cite{Hall} and  the orthogonal projection estimator 
\cite{L1}, \cite{L2}.
 Basharin \cite{Basharin} studied the plug-in entropy estimator for the finite value discrete case and obtained the mean, the variance and the central limit theorem of this estimator. Vu, Yu and Kass \cite{VYK} studied the coverage-adjusted entropy estimator with unobserved values for the infinite value discrete case. 

\section{Main results}

 Consider a random vector $Z = (X, Y)$. We call $Z$
a mixed-pair if $X\in \mathbb{R}$ is a discrete random variable with countable support $\mathcal{X}=\{x_1,x_2,...\}$ while $Y\in \mathbb{R}^d$ is a continuous random variable. Observe that $Z = (X, Y)$ induces measures $\{\mu_1, \mu_2, \cdots\}$ that are absolutely continuous with respect to
the Lebesgue measure, where $\mu_i(A) = \mathbb{P}(X = x_i, Y \in A)$, for every Borel set $A$ in $\mathbb{R}^d$. 
There  exists a non-negative function $g(x, y)$ with $h(x):=\int_{\mathbb{R}^d}g(x,y)dy$ be the probability mass function on $\mathcal{X}$ and $f(y):=\sum_i g_i(y)$ be the marginal density function of $Y$. Here,  $g_i(y)=g(x_i,y)$, $ i\in \mathbb{N}$. In particular, denote $p_i=h(x_i), i\in \mathbb{N}$. We have that $$f_i(y)=\frac{1}{p_i}g_i(y)$$ is the probability density function of $Y$ conditioned on $X = x_i$. In \cite{Nair}, the authors gave the following regulation of mixed-pair and then defined the joint entropy of a mixed-pair. 

\begin{definition} (Good mixed-pair). A mixed-pair random variables $Z = (X, Y )$ is called good if the following
condition is satisfied: $$\int_{\mathcal{X}\times\mathbb{R}^d}|g(x, y)\log g(x,y)|dxdy=\sum\limits_{i}\int_{\mathbb{R}^d}|g_i(y)\log g_i(y)|dy<\infty.$$ 
\end{definition}

Essentially, we have a good mixed-pair random variables when restricted to any of the $X$ values, the conditional differential entropy of $Y$ is well-defined.

\begin{definition} (Entropy of a mixed-pair). The entropy of a good mixed-pair random variable is defined by
$$H(Z)=-\int_{\mathcal{X}\times\mathbb{R}^d}g(x, y)\log g(x, y)dxdy=-\sum\limits_{i}\int_{\mathbb{R}^d}g_i(y)\log g_i(y)dy.$$
\end{definition}

As $g_i(y)=p_if_i(y)$ then we have that 

\begin{equation}\label{entropy}
\begin{split}
&H(Z)= -\sum\limits_{i}\int_{\mathbb{R}^d}g_i(y)\log g_i(y)dy \\
&=-\sum\limits_{i}\int_{\mathbb{R}^d}p_if_i(y)\log p_if_i(y)dy\\
&= -\sum\limits_{i}p_i\log p_i\int_{\mathbb{R}^d}f_i(y)dy-\sum\limits_{i}p_i\int_{\mathbb{R}^d}f_i(y)\log f_i(y)dy \\
&= -\sum\limits_{i}p_i\log p_i-\sum\limits_{i}p_i\int_{\mathbb{R}^d}f_i(y)\log f_i(y)dy\\
&=H(X)+\sum\limits_{i}p_iH(Y|X=x_i).
\end{split}
\end{equation}
 
We take the convention $\log 0=0$ and $\log 0/0=0$. From the general formula of the mutual information (\ref{MIdef}),  we get that
\begin{equation}\label{mutual}
\begin{split}
&I(X,Y)=\int_{\mathcal{X}\times \mathbb{R}^d}g(x, y)\log \frac{g(x, y)dxdy}{h(x)f(y)dxdy}dxdy \\
&=\sum\limits_{i}\int_{\mathbb{R}^d}g_i(y)\log \frac{g_i(y)}{p_if(y)}dy \\
&=\sum\limits_{i}\int_{\mathbb{R}^d} g_i(y)\log g_i(y)dy-\sum\limits_{i}\int_{\mathbb{R}^d} g_i(y)\log p_idy-\sum\limits_{i}\int_{\mathbb{R}^d} g_i(y)\log f(y)dy\\
&=\sum\limits_{i}\int_{\mathbb{R}^d} p_if_i(y)\log [p_if_i(y)]dy-\sum\limits_{i}p_i\log p_i\int_{\mathbb{R}^d} f_i(y)dy-\int_{\mathbb{R}^d} f(y)\log f(y)dy\\
&=\sum\limits_{i}p_i\log p_i\int_{\mathbb{R}^d} f_i(y)dy+\sum\limits_{i}p_i\int_{\mathbb{R}^d} f_i(y)\log f_i(y)dy-\sum\limits_{i}p_i\log p_i-\int_{\mathbb{R}^d} f(y)\log f(y)dy\\
&=-H(Z)+H(X)+H(Y)=H(Y)-\sum\limits_{i}p_iH(Y|X=x_i):=H(Y)-\sum_i I_i.
\end{split}
\end{equation} 

Let $(X,Y), (X_1,Y_1),...,(X_N,Y_N)$ be a random sample drawn from a mixed distribution with discrete component having support $\{0,1, \cdots, m\}$, and let $p_i=\mathbb{P}(X=i)$, $0\le i\le m$ with $0<p_i<1, \sum p_i=1$. Also suppose that the continuous component has pdf $f(y)$. Denote $\hat{p}_i=\sum\limits_{k=1}^{N}\mathbb{I}(X_k=i)/N$, $0\le i\le m,$ and let   
\begin{equation}\label{Ibar0}
\begin{split}
&\bar{I}_i = -\hat{p}_i\bigg[N\hat{p}_i\bigg]^{-1}\sum\limits_{k=1}^{N}\mathbb{I}(X_k=i)\log f_i(Y_k) \\
&= -N^{-1}\sum\limits_{k=1}^{N}\mathbb{I}(X_k=i)\log f_i(Y_k)
\end{split}
\end{equation}
and 
\begin{equation}\label{Hbar}
\begin{split}
&\bar{H}(Y) = - N^{-1}\sum\limits_{k=1}^{N}\log f(Y_k)
\end{split}
\end{equation}
be  the estimators of $I_i=p_iH(Y|X=i)$, $0\le i\le m$, and $H(Y)$ respectively, where $f_i(y)$ is the probability density function of $Y$ conditioned on $X = i$, $0\le i\le m$. Denote $a=(1,-1,\cdots, -1)^\intercal$. Let  $\Sigma$ 
be the covariance matrix of $(\log f(Y), \mathbb{I}(X=0)\log f_0(Y), \cdots,\mathbb{I}(X=m)\log f_m(Y))^\intercal$.
\begin{theorem}\label{thm1}
$a^\intercal\Sigma a>0 $ if and only if $X$ and $Y$ are dependent. 
For the estimator 
\begin{equation}\label{muestimator}
\bar{I}(X,Y)=\bar{H}-\sum_{i=0}^m\bar{I}_i
\end{equation}
of $I(X,Y)$ we have that 
\begin{equation}\label{clt1}
\sqrt{N}(\bar{I}(X,Y)-I(X,Y))\to N(0,a^\intercal\Sigma a)
\end{equation}
given that $X$ and $Y$ are dependent.
Furthermore, the variance $a^\intercal\Sigma a$ can be calculated by 
\begin{equation}\label{sigma}
\begin{split}
& a^\intercal\Sigma a= var\big(\log f(Y)\big)+\sum_{i=0}^m p_i E_i[\log f_i(Y)]^2-\sum_{i=0}^m p_i^2\big(E_i[\log f_i(Y)]\big)^2\\
&-2 \sum_{i=0}^m p_i[E_i\log f_i(Y)\log f(Y)-E_i\log f_i(Y) E\log f(Y)]\\
&-2\sum_{0\le i<j\le m}p_ip_j[E_i\log f_i(Y)][E_j\log f_j(Y)],
\end{split}
\end{equation}
where $E_i$ is the conditional expectation of $Y$ given $X=i, 0\le i\le m$.
\end{theorem}
\begin{proof}
First of all, $a^\intercal\Sigma a\ge 0$ since $\Sigma$ is the variance covariance matrix. If $a^\intercal\Sigma a=0$ then 
\begin{equation*}
var\left(\log f(Y)-\sum_{i=0}^m \mathbb{I}(X=i)\log f_i(Y)\right)=a^\intercal\Sigma a=0
\end{equation*}
and  $\log f(Y)-\sum_{i=0}^m \mathbb{I}(X=i)\log f_i(Y)\equiv C$  for some constant $C$.  But 
\begin{equation*}
\log f(Y)-\sum_{i=0}^m \mathbb{I}(X=i)\log f_i(Y)
=\sum_{i=0}^m \mathbb{I}(X=i)\log \frac{f(Y)}{f_i(Y)}.
\end{equation*}
Hence $\log \frac{f(Y)}{f_i(Y)}\equiv C$. Then  $f_i(y)=cf(y)$ for some constant $c>0$ and for all $0\le i\le m$. But 
$f(y)=\sum_{i=0}^m p_if_i(y)=cf(y)\sum_{i=0}^m p_i=cf(y)$. Hence, $c\equiv 1$ and  $f_i(y)=f(y)$ for all $0\le i\le m$. Then $X$ and $Y$ are independent.
On the other hand, if $X$ and $Y$ are independent, then $f_i(y)=f(y)$  for all $0\le i\le m$. Therefore, $\log f(Y)-\sum_{i=0}^m \mathbb{I}(X=i)\log f_i(Y)=0$ and $a^\intercal\Sigma a=0$. Hence, $a^\intercal\Sigma a=0$ if and only if $X$ and $Y$ are independent.

Notice that the vector $(\bar{H}(Y), \bar{I}_0,\cdots, \bar{I}_m)^\intercal$ is the sample mean of a sequence of i.i.d. random vectors $$\left\{(\log f(Y_k), \mathbb{I}(X_k=0)\log f_0(Y_k), \cdots, \mathbb{I}(X_k=m)\log f_m(Y_k))^\intercal\right\}_{k=1}^N$$ with mean $(H(Y), I_0,\cdots, I_m)^\intercal$. Then, by central limit theorem,  we have 
$$\sqrt{N}\left (\begin{pmatrix} \bar{H} \\ \bar{I}_0 \\ \vdots\\ \bar{I}_m\end{pmatrix}-\begin{pmatrix} H \\ I_0 \\ \vdots\\ I_m\end{pmatrix}\right )\to N( \bar{0},\Sigma ),$$
and, given $a^\intercal\Sigma a>0$, we have (\ref{clt1}). 
By the formula for variance decomposition, we have 
\begin{eqnarray}\label{varI0}
\begin{split}
&var\big(\mathbb{I}(X=i)\log f_i(Y)\big)\\
&= E\big\{var[\mathbb{I}(X=i)\log f_i(Y)|X]\big\}+var\big\{E[\mathbb{I}(X=i)\log f_i(Y)|X]\big\}\\
&= E\big\{\mathbb{I}(X=i) var[\log f_i(Y)|X]\big\}+var\big\{\mathbb{I}(X=i)E[\log f_i(Y)|X]\big\}\\
&= E\big\{\mathbb{I}(X=i) \sum_{j=0}^m var_j(\log f_j(Y))\mathbb{I}(X=j)\big\}\\
&\;\;\;\;+var\big\{\mathbb{I}(X=i)\sum_{j=0}^m E_j(\log f_j(Y))\mathbb{I}(X=j)\big\}\\
&= var_i[\log f_i(Y)] E\big\{\mathbb{I}(X=i) \big\}+\big(E_i[\log f_i(Y)]\big)^2var\big\{\mathbb{I}(X=i)\big\}\\
&= p_i var_i[\log f_i(Y)]+(p_i-p_i^2)\big(E_i[\log f_i(Y)]\big)^2\\
&=p_i E_i[\log f_i(Y)]^2-p_i^2\big(E_i[\log f_i(Y)]\big)^2,
\end{split}
\end{eqnarray}
$0\le i\le m$. Here $var_i$ is the conditional variance of $Y$ when $X=i, 0\le i\le m$.
By similar calculation,
\begin{eqnarray}\label{cov1}
\begin{split}
&Cov\bigg(\mathbb{I}(X=i)\log f_i(Y),\mathbb{I}(X=j)\log f_j(Y)\bigg)\\
&=-p_ip_j[E_i\log f_i(Y)][E_j\log f_j(Y)],
\end{split}
\end{eqnarray}
for all $0\le i<j\le m$, and 
\begin{equation}\label{cov2}
\begin{split}
&Cov\bigg(\mathbb{I}(X=i)\log f_i(Y),\log f(Y)\bigg)\\
&=p_i[E_i\log f_i(Y)\log f(Y)-E_i\log f_i(Y) E\log f(Y)].
\end{split}
\end{equation}
Thus, the covariance matrix $\Sigma$ 
of $(\log f(Y), \mathbb{I}(X=0)\log f_0(Y), \cdots,\mathbb{I}(X=m)\log f_m(Y))^\intercal$ and therefore $a^\intercal\Sigma a$ can be calculated by the above calculation (\ref{varI0})-(\ref{cov2}). We then have (\ref{sigma}). 
\end{proof}

We consider the case when the random variables $X$ and $Y$ are dependent. Note that in this case $a^\intercal\Sigma a>0$ and we have \eqref{clt1}. 
However, $\bar{I}(X,Y)$ is not a practical estimator since the density functions involved are not known. 

Now let $K(\cdot)$ be a kernel function in $\mathbb{R}^d$ and let $h$ be the bandwidth. Then
$$\hat{f}_{ik}(y)=\bigg\{(N\hat{p}_i-1)h^d\bigg\}^{-1}\sum\limits_{j\neq k}\mathbb{I}(X_j=i)K\{(y-Y_j)/h\}$$
are the ``leave-one-out" estimators of the functions $f_i$, $0\le i\le m$, and  
\begin{equation}\label{Ihat0}
\begin{split}
&\hat{I}_i=- N^{-1}\sum\limits_{k=1}^{N}\mathbb{I}(X_k=i)\log \hat{f}_{ik}(Y_k) \\
\end{split}
\end{equation}
are estimators of $I_i=p_iH(Y|X=i)$, $0\le i\le m$. Also 
\begin{equation}\label{Hhat}
\begin{split}
&\hat{H}=- N^{-1}\sum\limits_{k=1}^{N}\log \hat{f}_k(Y_k) \\
\end{split}
\end{equation}
is an estimator of $H(Y)$, where
\begin{equation}\label{fhat}
\begin{split}
&\hat{f}_k(y)=\bigg\{(N-1)h^d\bigg\}^{-1}\sum\limits_{j\neq k}K\{(y-Y_j)/h\} \\
&=\bigg\{(N-1)h^d\bigg\}^{-1}\sum\limits_{j\neq k}[\sum_{i=0}^m\mathbb{I}(X_k=i)]K\{(y-Y_j)/h\}\\
&= \sum_{i=0}^m\frac{N\hat{p}_i-1}{N-1}\hat{f}_{ik}(y).
\end{split}
\end{equation}
\begin{theorem}
Assume that the tails of $f_0, \cdots, f_m$ are decreasing like $|x|^{-\alpha_0}, \cdots, |x|^{-\alpha_m}$ , respectively, as $|x|\to\infty$. 
Also assume that the kernel function has appropriately heavy tails  as in  \cite{Hall1}. If $h=o(N^{-1/8})$ and $\alpha_0\cdots, \alpha_m$ are all greater than $7/3$ in the case $d=1$, greater than $6$ in the case $d=2$ and greater than $15$ in the case $d=3$, then for the estimator 
\begin{equation}\label{Ihat}
\hat{I}(X,Y)=\hat{H}-\sum_{i=0}^m\hat{I}_i, 
\end{equation}
we have
\begin{equation}\label{clt2}
\sqrt{N}(\hat{I}(X,Y)-I(X,Y))\to N(0,a^\intercal\Sigma a).
\end{equation}
\end{theorem}
\begin{proof}
Under the conditions in the theorem,  applying the formula (3.1) or (3.2) from \cite{Hall}, we have
 $$\hat{H}=\bar{H}+o(N^{-1/2}), \quad \hat{I}_0=\bar{I}_0+o(N^{-1/2}), \cdots, \quad  \hat{I}_m=\bar{I}_m+o(N^{-1/2}).$$
Together with Theorem \ref{thm1}, we have (\ref{clt2}). 
\end{proof}

We may take the probability density function of Student-$t$ distribution with proper degree of freedom instead of the normal density function as the kernel function. On the other hand, 
if $X$ and $Y$ are independent then $I(X,Y)=\bar{I}(X,Y)=0$ and we have that $\hat{I}(X,Y)=o(N^{-1/2})$.

\section{Simulation study}
In this section we conduct a simulation study with $m=1$, i.e., the random variable $X$ takes two possible values 0 and 1, to confirm the main results stated in (\ref{clt2}) for the kernel mutual information estimation of good mixed-pairs. First we study some one dimensional examples. Let $t(\nu, \mu, \sigma)$ be the Student t distribution with degree of freedom $\nu$, location parameter $\mu$ and scale parameter $\sigma$ and let $pareto(x_m, \alpha)$ be the Pareto distribution with density function $f(x)=\alpha x_m^\alpha x^{-(\alpha+1)}\mathbb{I}(x\ge x_m)$.  We study the mixture for the following four cases: 1). $t(3, 0, 1)$ and $t(12, 0, 1)$;   2).  $t(3, 0, 1)$ and $t(3, 2, 1)$;  3). $t(3, 0, 1)$ and $t(3, 0, 3)$; 4). $pareto(1,2)$ and $pareto(1,10)$. For each case, $p_0=0.3$ for the first distribution and $p_1=0.7$ for the second distribution.  

The second row of Table \ref{tab}  lists the mathematica calculation of the mutual information (MI) as stated in (\ref{mutual})  for each case. The third row of Table \ref{tab} gives the average of 400 estimates based on formula (\ref{Ihat}). For each estimate, we use the probability density function of the Student t distribution with degree of freedom 3, i.e. $t(3, 0, 1)$, as the kernel function. We also have simulation study with kernel functions satisfying the conditions in the main results and obtained similar results. We take $h=N^{-1/5}$ as the bandwidth for the first three cases and $h=N^{-1/5}/24$ for the last case. The data size for each estimate is $N=50,000$ in each case. The Pareto distributions $pareto(1,2)$ and $pareto(1,10)$ have very dense area on the right of 1. This is the reason that we take a relatively small bandwidth for this case. To apply the kernel method in estimation, one should select an optimal bandwidth based on some criteria, for example, to minimize the mean squared error.  It is interesting to investigate the bandwidth selection problem from both theoretical and application viewpoints. However, it seems that the study in this direction is very difficult. We leave it as an open question for future study. It is clear that the average of the estimates matches the true value of mutual information. 

We apply mathematica to calculate the covariance matrix $\Sigma$
of 
$$(\log f(Y), \mathbb{I}(X=0)\log f_0(Y),\mathbb{I}(X=1)\log f_1(Y))^\intercal$$
 and, therefore, the value of $a^\intercal \Sigma a$ for each case by formulae  (\ref{varI0})-(\ref{cov2}) or (\ref{sigma}). The values of $a^\intercal \Sigma a$ are $0.02189236$, $0.3092179$, $0.1540501$ and $0.2748102$ respectively for the four cases. 
The fourth row of Table \ref{tab} lists the values of $(a^\intercal \Sigma a/N)^{1/2}$ which serves as the asymptotic approximation of the standard deviation of the estimator $\hat{I}(X,Y)$ in the central limit theorem (\ref{clt2}). The last row gives the sample standard deviation from $M=400$ estimates. These two values also have good match.
 \begin{table}[H]
\center
\bigskip
\begin{tabular}{|l |c|c|c|c|c}
\hline\hline                                                     
\;\;mixture   &$t(3, 0, 1)$    &$t(3, 0, 1)$   &  $t(3, 0, 1)$  &  $pareto(1,2)$ \\
                   &$t(12, 0, 1)$  &$t(3, 2, 1)$   & $t(3, 0, 3)$   & $pareto(1,10)$\\ 
\hline

\;\;MI           & 0.011819 & 0.20023 & 0.102063 & 0.201123\\
\hline

\;\;mean of estimates & 0.01167391 & 0.1991132 & 0.1014199 & 0.2010447\\
\hline

\;\;$(a^\intercal\Sigma a/N)^{1/2}$   & 0.0006617 & 0.0025 & 0.0018 & 0.0023\\
\hline

\;\;sample sd          & 0.0006616724 & 0.002345997 & 0.001819982 & 0.002349275\\

\hline\hline
\end{tabular}  
\caption{ True value of the mutual information and the mean value of the estimates.} 
\label{tab}
\end{table}

%
\begin{figure}[H]
\centering
\begin{tabular}{cc}
\includegraphics[width=0.33\linewidth]{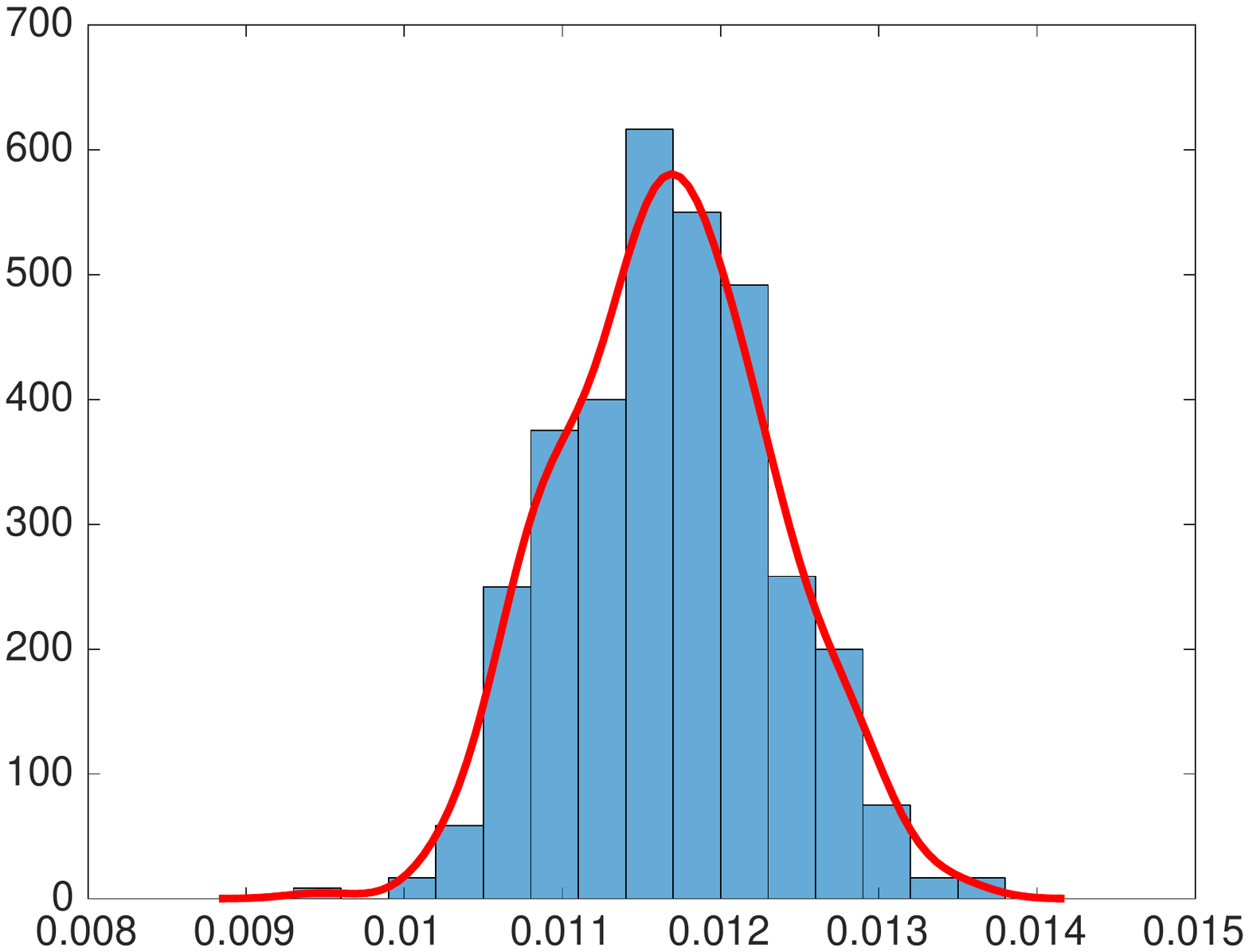} &
\includegraphics[width=0.33\linewidth]{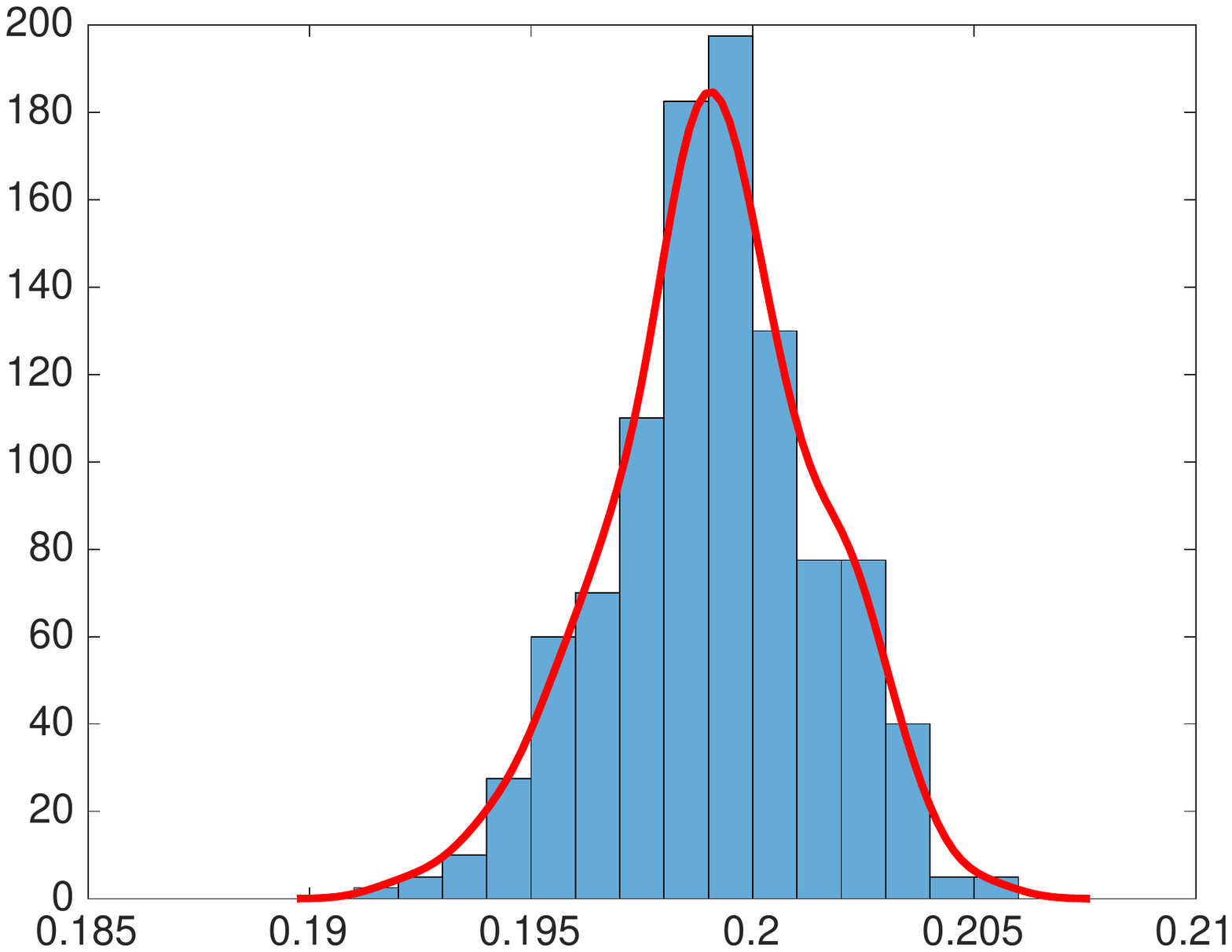}\vspace{-3cm}\\
\includegraphics[width=0.33\linewidth]{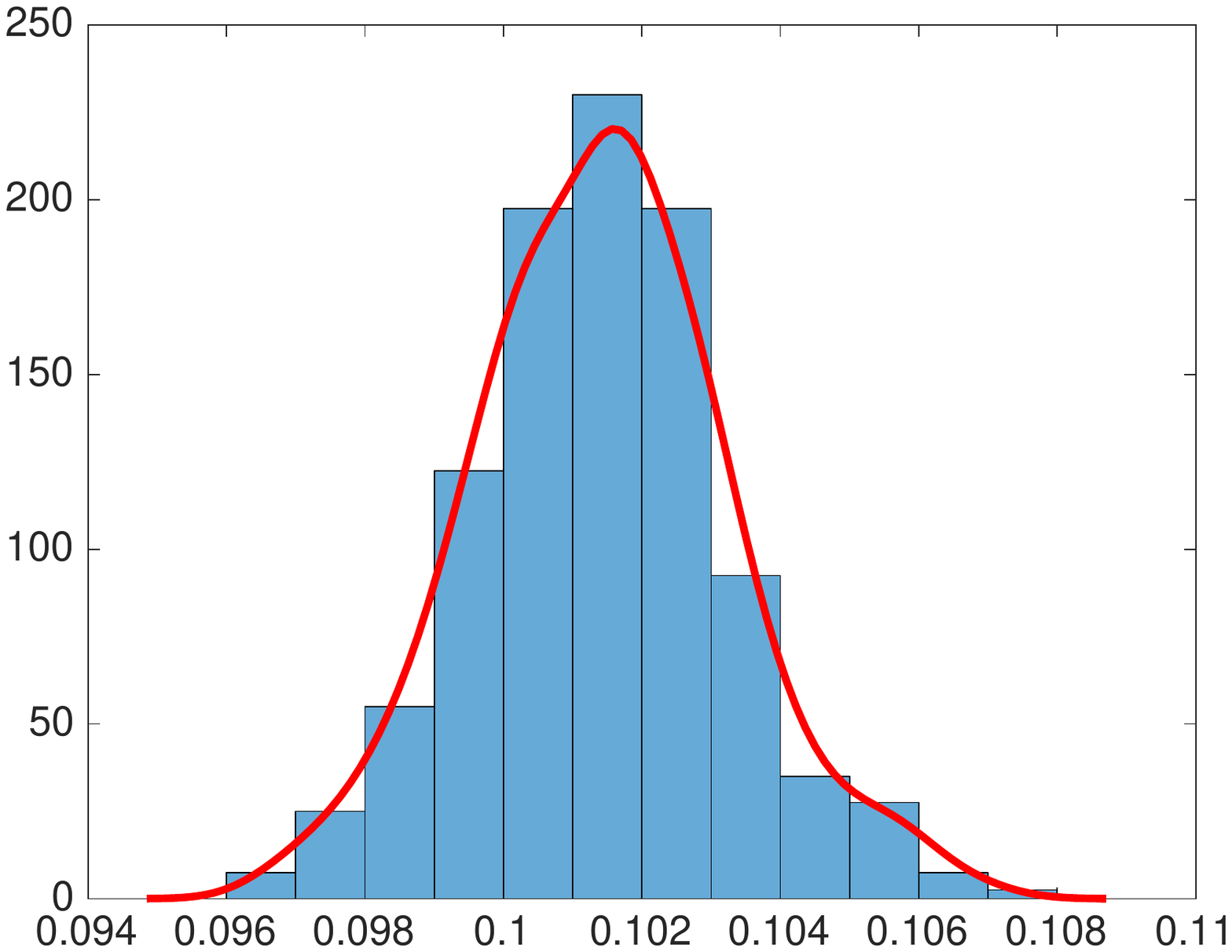} &
\includegraphics[width=0.33\linewidth]{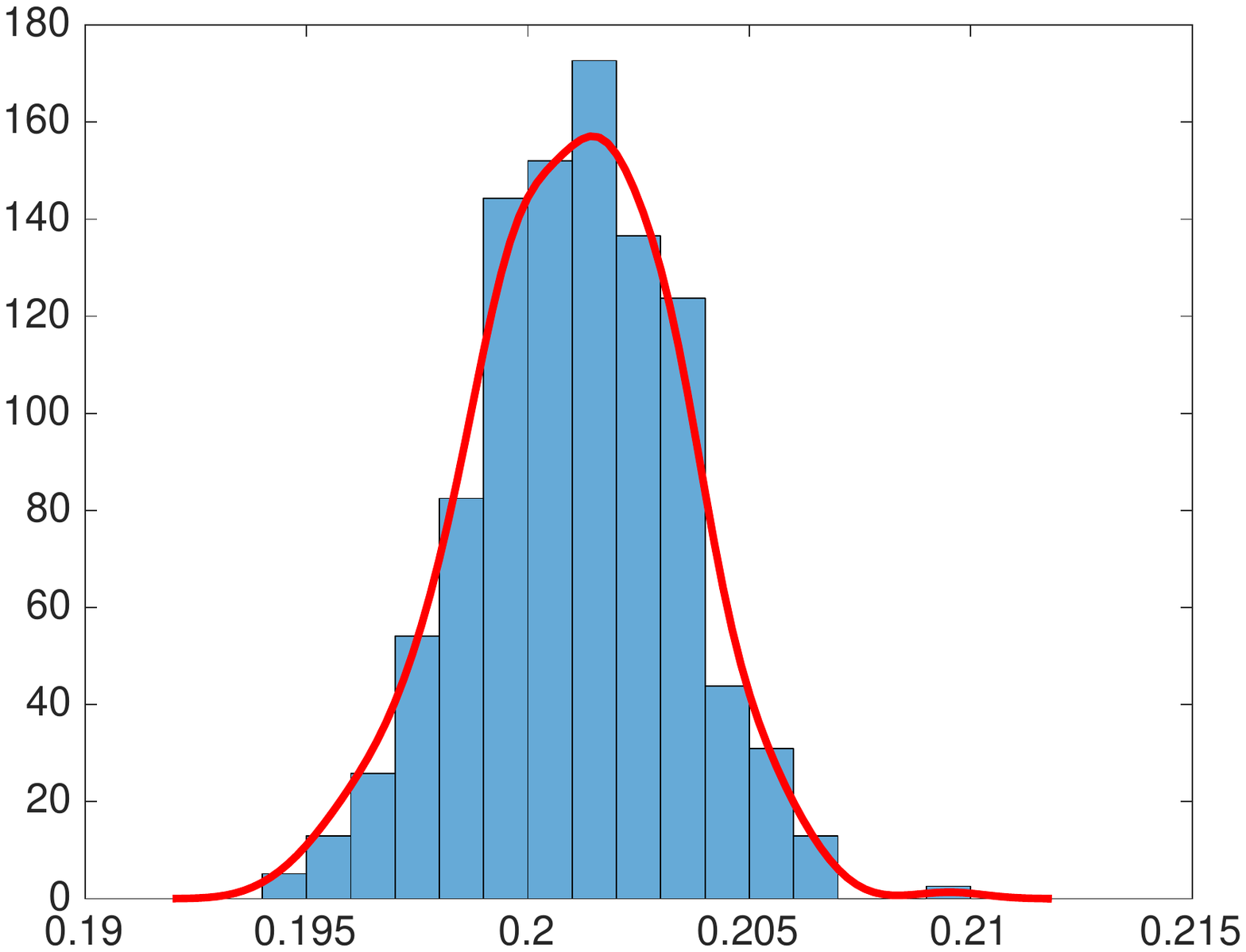}\vspace{-1cm}\\
\end{tabular}
\caption {The histograms with kernel density fits of $M=400$ estimates. Top left: $t(3,0,1)$ and $t(12,0,1)$. Top right: $t(3,0,1)$ and $t(3,2,1)$. 
Bottom left: $t(3,0,1)$ and $t(3,0,3)$. Bottom right: $pareto(1,2)$ and $pareto(1,10)$. \label{fig:1}}
\end{figure}
\begin{figure}[H]
\centering
\begin{tabular}{cc}
\includegraphics[width=0.33\linewidth]{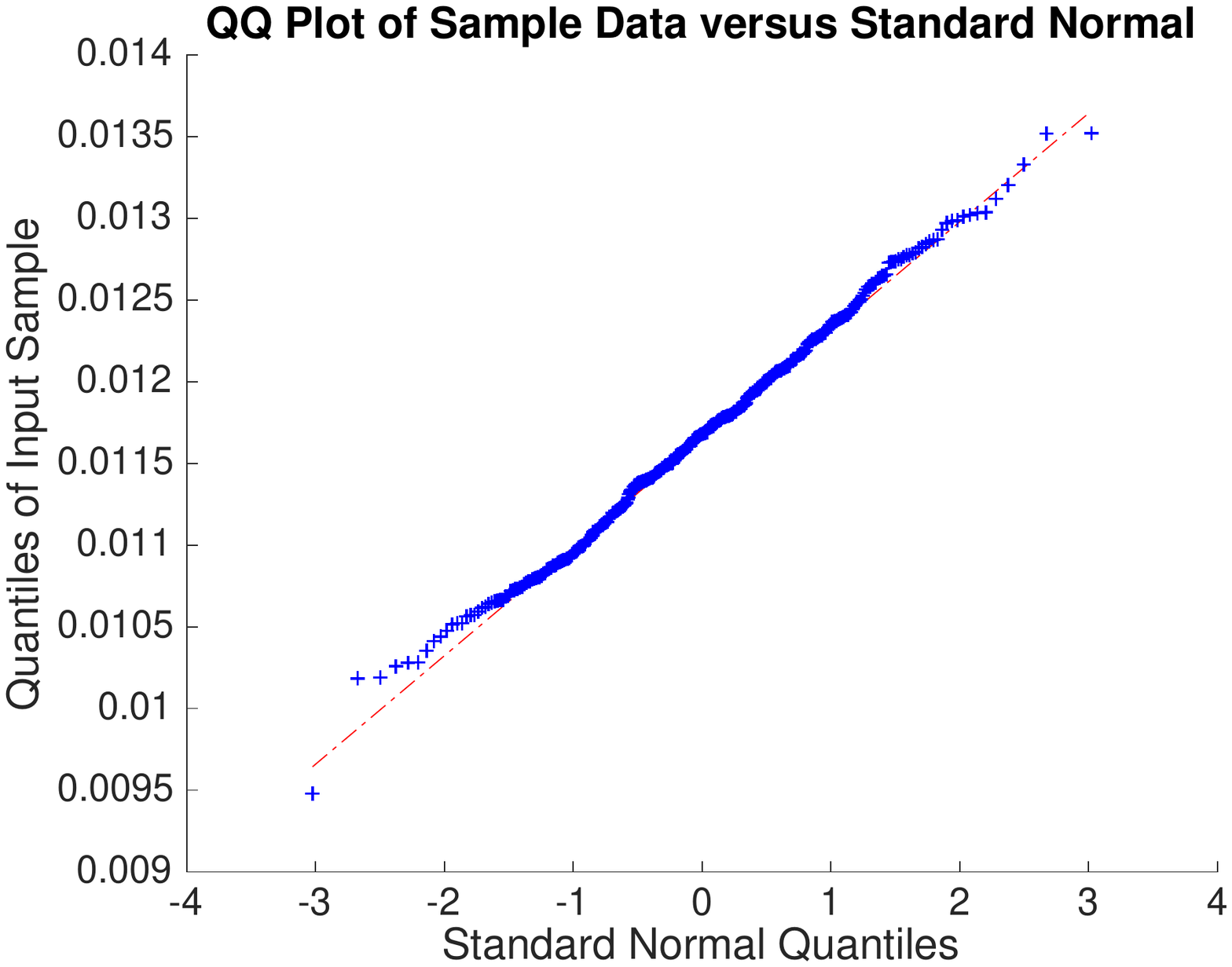}
\includegraphics[width=0.33\linewidth]{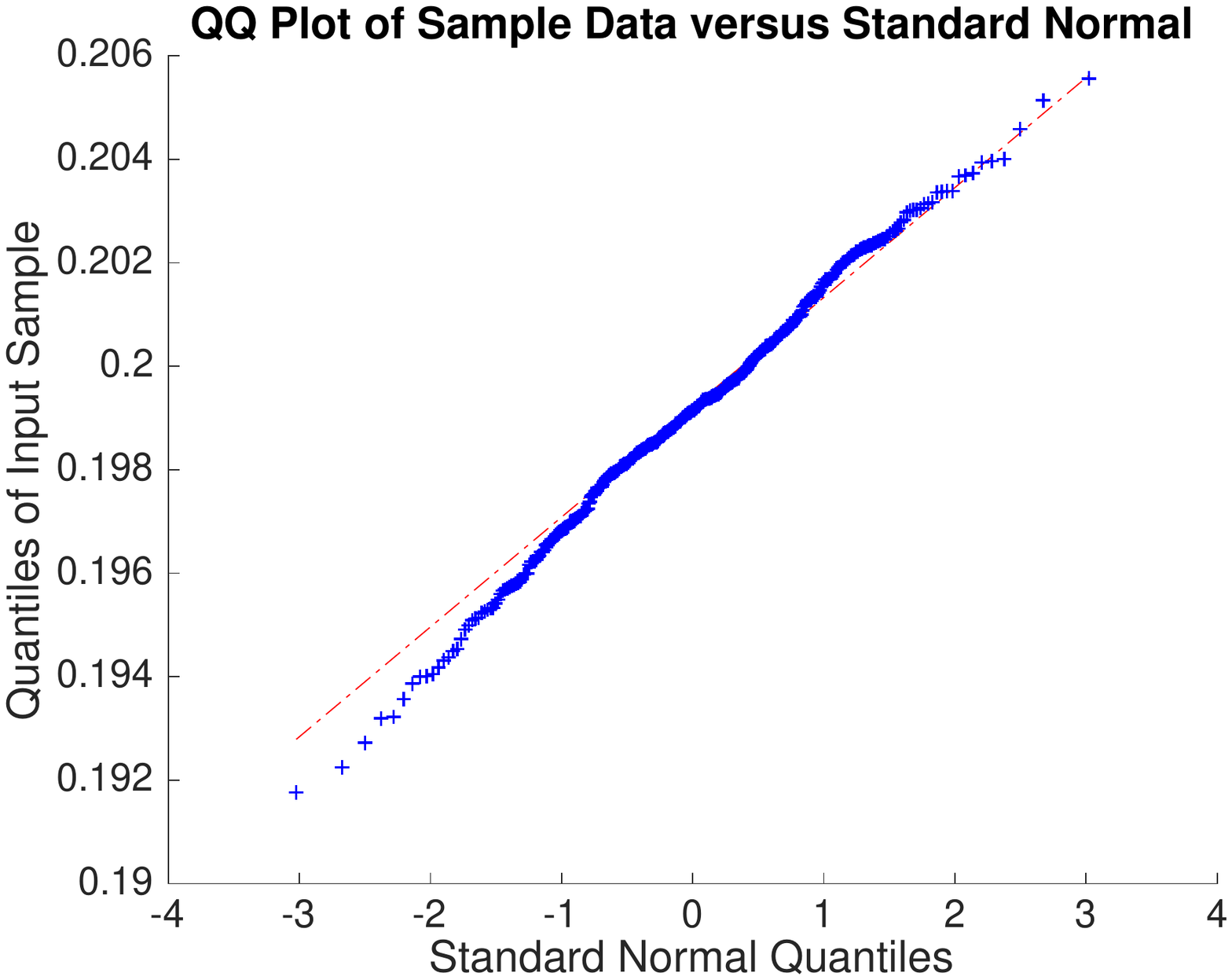}\vspace{-3cm}\\
\includegraphics[width=0.33\linewidth]{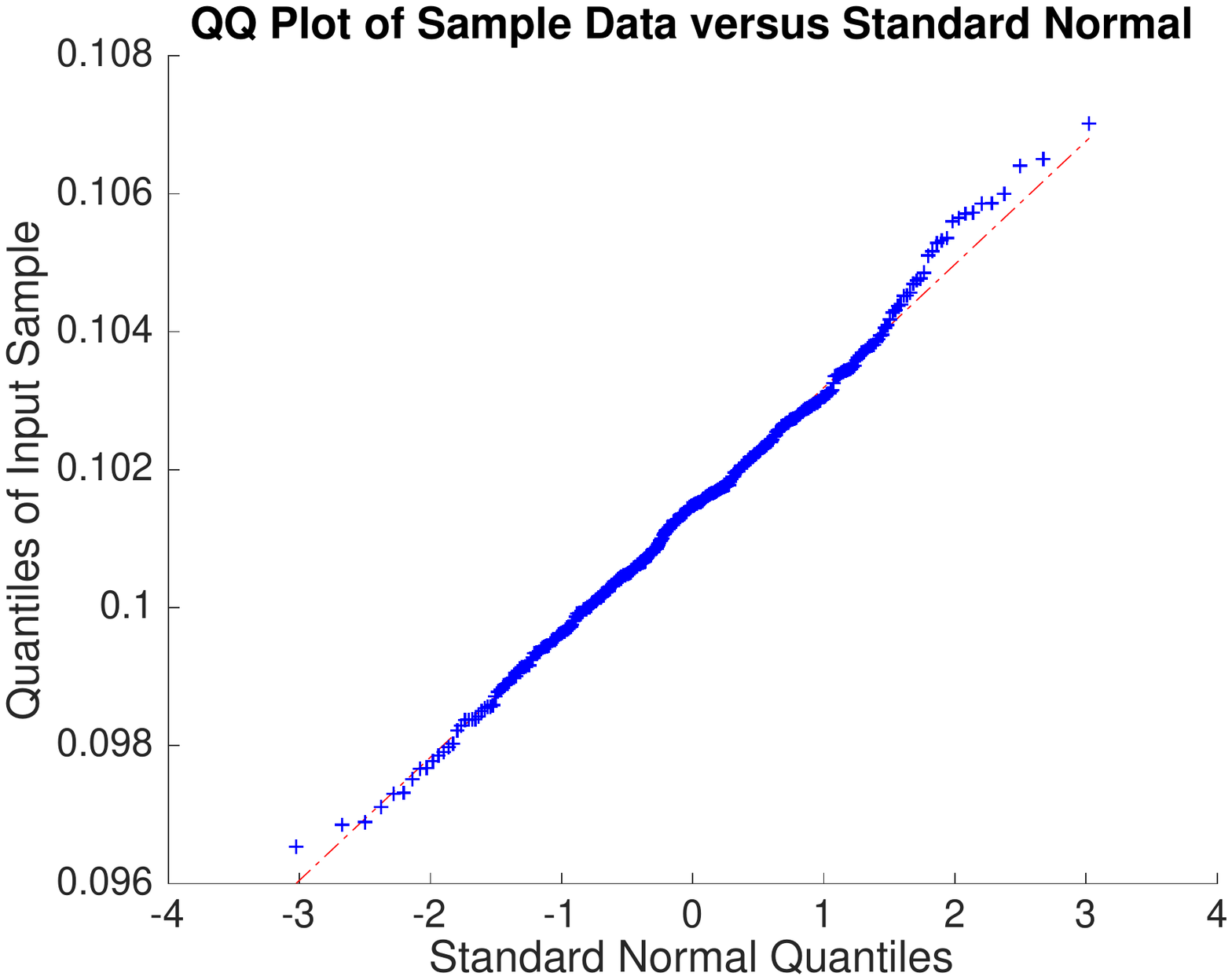}
\includegraphics[width=0.33\linewidth]{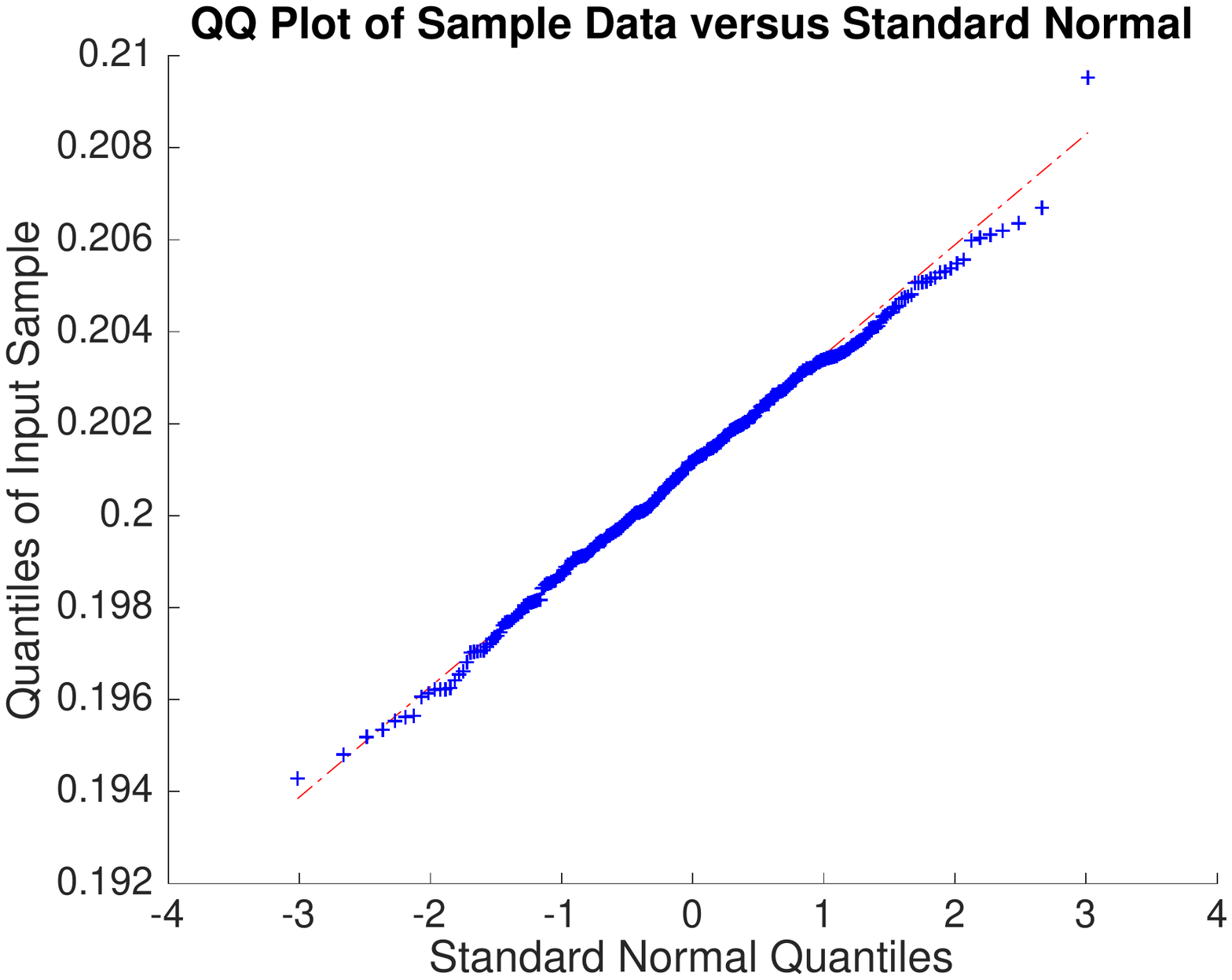}
\end{tabular}
\caption {The Q-Q plots of $M=400$ estimates. Top left: $t(3,0,1)$ and $t(12,0,1)$. Top right: $t(3,0,1)$ and $t(3,2,1)$. 
Bottom left: $t(3,0,1)$ and $t(3,0,3)$. Bottom right: $pareto(1,2)$ and $pareto(1,10)$.\label{fig:2}}
\end{figure}

 Figure \ref{fig:1} and \ref{fig:2} show the histograms with kernel density fits and normal Q-Q plots of 400 estimates for each case. It is clear that the values of $\hat{I}(X,Y)$ follow a normal distribution. 

We study two examples in the two dimensional case. Let $t_\nu(\mu,\Sigma_0)$ be the two dimensional Student t distribution with degree of freedom $\nu$, mean $\mu$ and shape matrix $\Sigma_0$. We study the mixture in two cases: 1). $t_5(0, I)$ and $t_{25}(0, I)$;   2).  $t_5(0, I)$ and $t_5(0, 3I)$. Here $I$ is the identity matrix. For each case, $p_0=0.3$ for the first distribution and $p_1=0.7$ for the second distribution. Table \ref{tab2} summarizes $200$ estimates of the mutual information with $h=N^{-1/5}$ and sample size $N=50,000$ for each estimate.  We take $t_3(0,I)$ as the kernel function. Same as the one dimensional case, we apply mathematica to calculate the true value of MI and $(a^\intercal\Sigma a/N)^{1/2}$ which is given in formula (\ref{sigma}).  Figure \ref{fig:3} shows the histograms with kernel density fits and normal Q-Q plots of 200 estimates for each example. It is clear that the values of $\hat{I}(X,Y)$ also follow a normal distribution in the two dimensional case. In summary, the simulation study confirms the central limit theorem as stated in (\ref{clt2}). 

 \begin{table}[H]
\center
\bigskip
\begin{tabular}{|l |c|c|c}
\hline\hline                                                     
\;\;mixture   &$t_5(0, I)$    &$t_5(0, I)$    \\
                   &$t_{25}(0, I)$  &$t_5(0, 3I)$   \\ 
\hline

\;\;MI           & 0.01158 & 0.202516 \\
\hline

\;\;mean of estimates & 0.0112381 & 0.2022715 \\
\hline

\;\;$(a^\intercal\Sigma a/N)^{1/2}$   & 0.0006577826 & 0.002312909 \\
\hline

\;\;sample sd          & 0.0008356947 & 0.002315134 \\

\hline\hline
\end{tabular}  
\caption{ True value of the mutual information and the mean value of the estimates.} 
\label{tab2}
\end{table}

\begin{figure}[H]
\centering
\begin{tabular}{cc}
\includegraphics[width=0.33\linewidth]{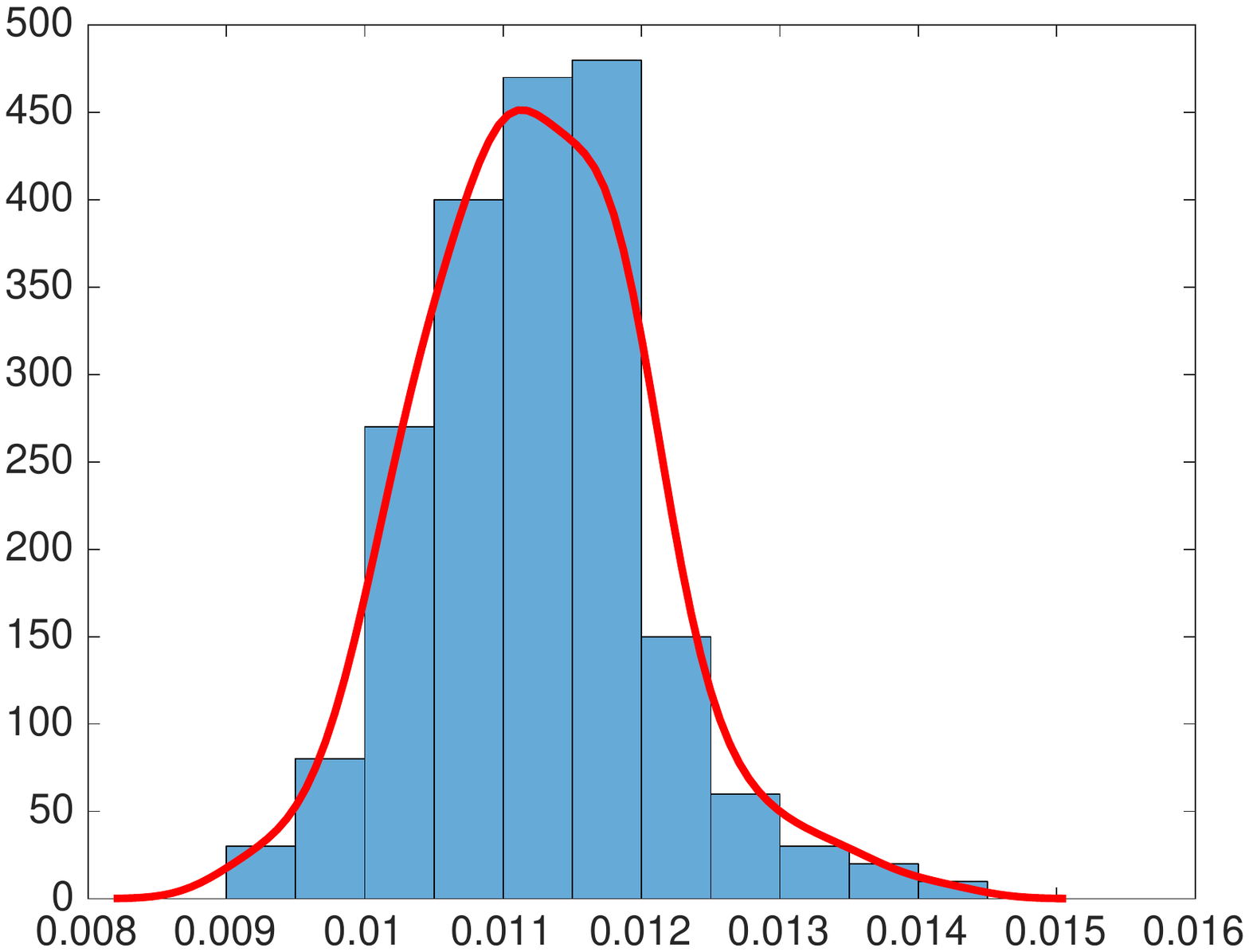}
\includegraphics[width=0.33\linewidth]{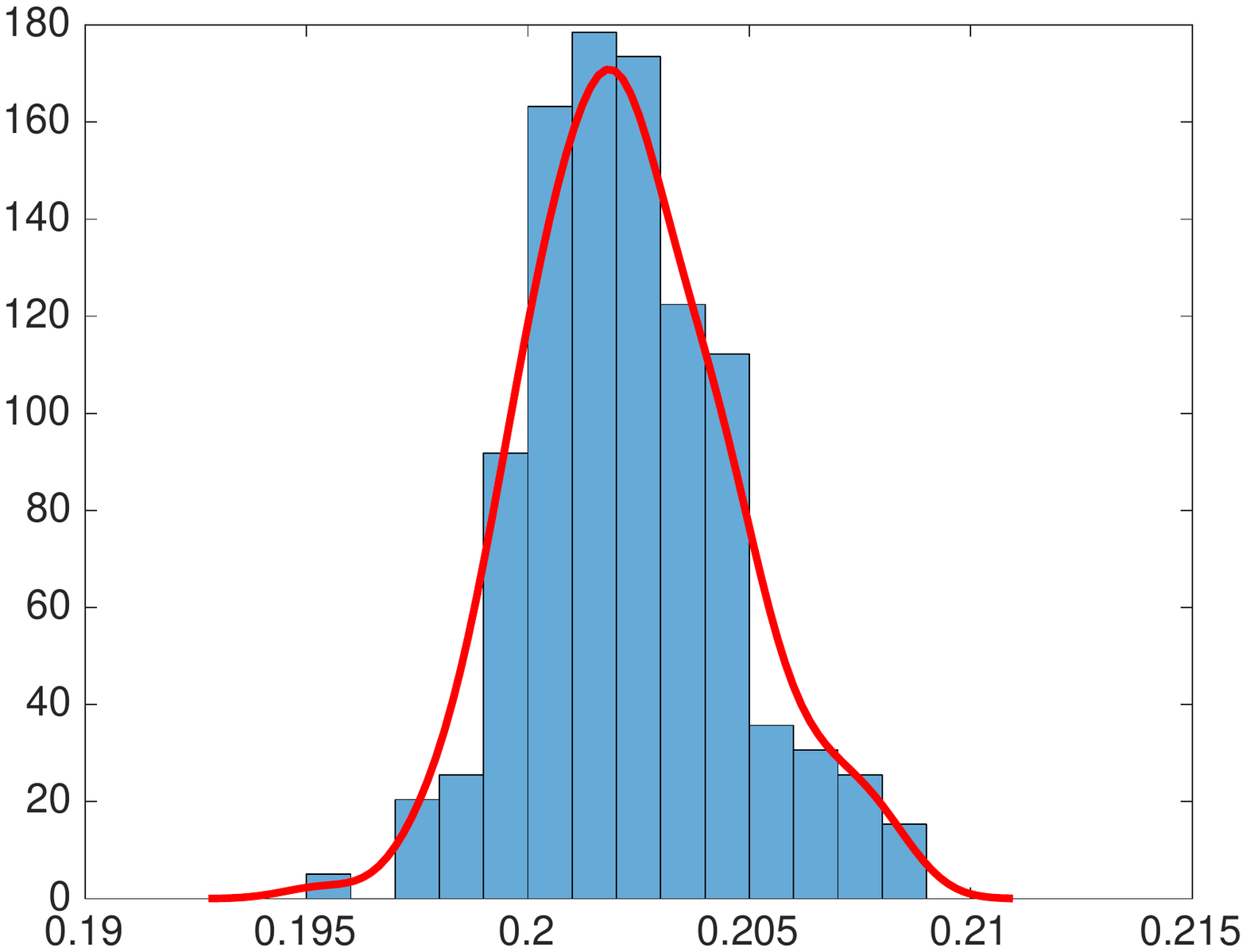}\vspace{-3cm}\\
\includegraphics[width=0.33\linewidth]{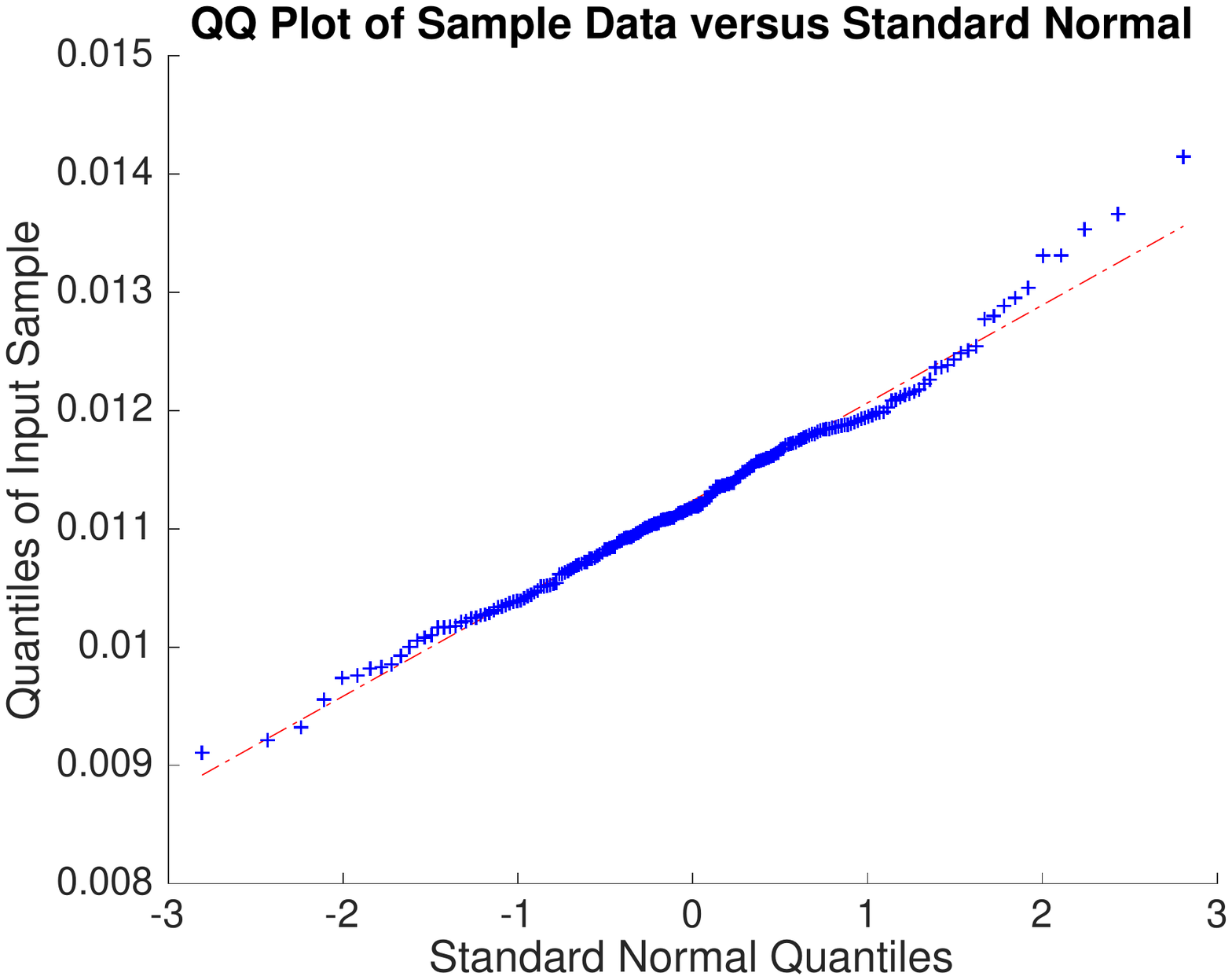}
\includegraphics[width=0.33\linewidth]{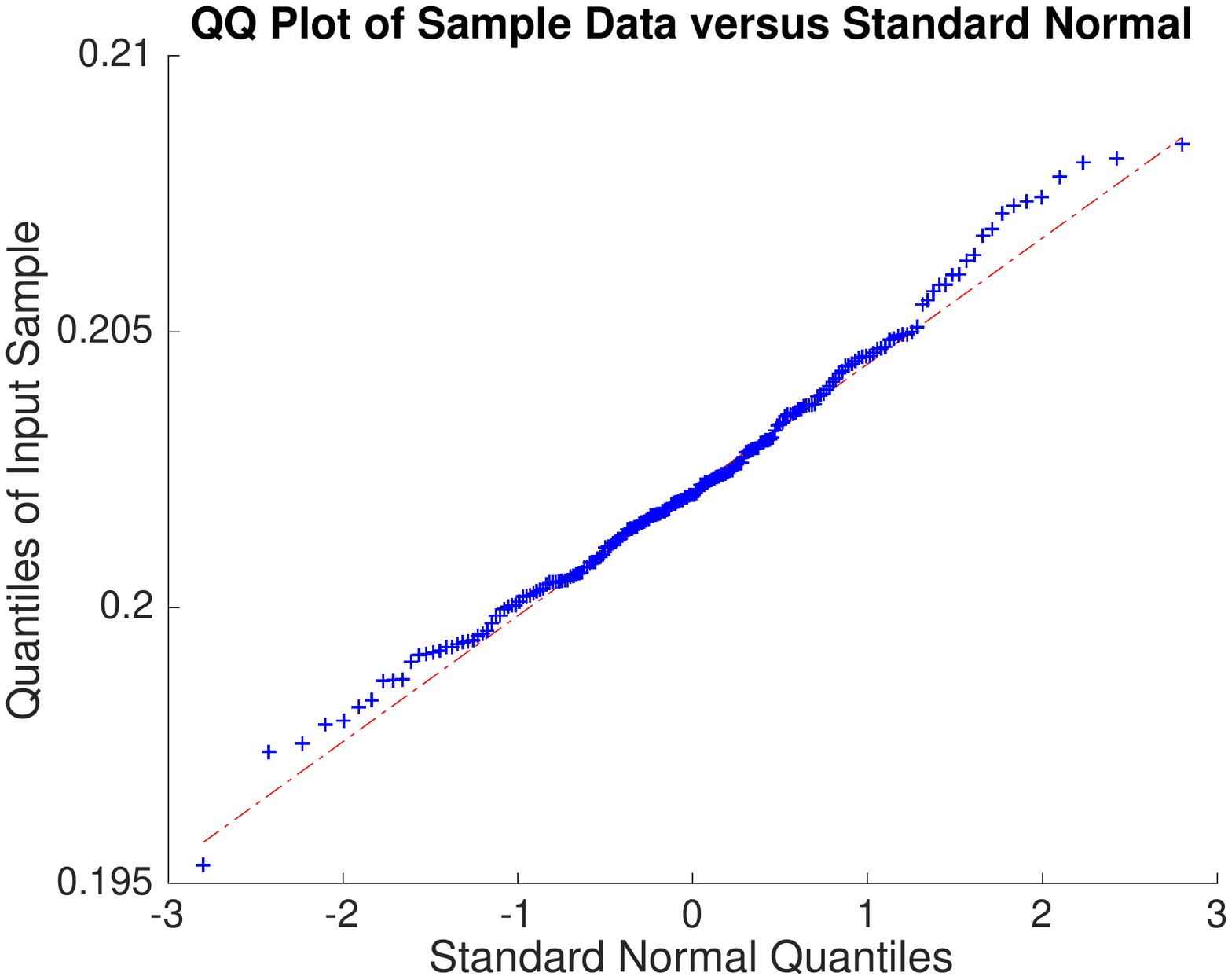}
\end{tabular}\vspace{-1.5cm}
\caption {The histograms and Q-Q plots of $M=200$ estimates.  Left: $t_5(0, I)$ and $t_{25}(0, I)$.   Right:  $t_5(0, I)$ and $t_5(0, 3I)$. \label{fig:3}}
\end{figure}

\noindent \textbf{Acknowledgement}\\

The authors thank the editor and the referees for carefully reading the manuscript and for the suggestions that improved the presentation. This research is supported by the College of Liberal Arts
Faculty Grants for Research and Creative Achievement at the University of Mississippi. The research of Hailin Sang is also supported by the Simons Foundation Grant 586789.

\end{document}